\documentclass[11pt]{amsart}

\usepackage[margin=2cm]{geometry}

\usepackage{latexsym, amssymb, amsfonts, amsmath, amsthm, graphics, graphicx, subfigure, bbm, stmaryrd,xcolor}
\usepackage{epstopdf}
\usepackage[pdftex,colorlinks,linkcolor=blue,menucolor=red,backref=false,bookmarks=true,citecolor=olive]{hyperref}

\usepackage{palatino}

\newtheorem{theorem}{Theorem}

\newtheorem{corollary}[theorem]{Corollary}

\newtheorem{definition}[theorem]{Definition}

\newtheorem{lemma}[theorem]{Lemma}

\newtheorem{problem}[theorem]{Problem}

\newcommand{\field}[1]{\mathbb{#1}}
\newcommand{\R}{\field{R}}

\newcommand{\E}{\field{E}}

\newcommand{\p}{\field{P}}

\newcommand{\tr}{\mathrm{tr}}

\newcommand{\met}[2][ccccccccccccccccccccccccccccccccc]{\left[ \begin{array}{#1} #2 \\ \end{array}\right]}

\numberwithin{equation}{section}

\begin{document}

\title{Shy and Fixed-Distance Couplings of Brownian Motions on Manifolds}
\author{Mihai N. Pascu}
\author{Ionel Popescu}
\address{"Transilvania" University of Bra\c sov \\
Faculty of Mathematics and Computer Science \\
Str. Iuliu Maniu Nr. 50 \\
Brasov -- 500091 \\
ROMANIA }
\email{mihai.pascu@unitbv.ro}
\address{School of Mathematics, Georgia Institute of Technology, 686 Cherry Street, Atlanta, GA 30332, USA \and  ``Simion Stoilow'' Institute of Mathematics   of Romanian Academy, 21 Calea Grivi\c tei, Bucharest, ROMANIA}
\email{ipopescu@math.gatech.edu,  ionel.popescu@imar.ro}

\begin{abstract}
In this paper we introduce three Markovian couplings of Brownian motions on smooth Riemannian  manifolds without boundary which sit at the crossroad of two concepts.   The first concept is the one of shy coupling put forward in \cite{Burdzy-Benjamini} and the second concept is the lower bound on the Ricci curvature and the connection with couplings made in \cite{ReSt}.

The first construction is the shy coupling, the second one is a fixed-distance coupling and the third is a coupling in which the distance between the processes is a deterministic exponential function of time.

The result proved here is that an arbitrary Riemannian manifold satisfying some technical conditions supports shy couplings.  If in addition, the Ricci curvature is non-negative, there exist fixed-distance couplings.  Furthermore, if the Ricci curvature is bounded below by a positive constant, then there exists a coupling of Brownian motions for which the distance between the processes is a decreasing exponential function of time.  The constructions use the intrinsic geometry, and relies on an extension of the notion of frames which plays an important role for even dimensional manifolds.

In fact, we provide a wider class of couplings in which the distance function is deterministic in Theorem \ref{t:100} and Corollary~\ref{Cor:9}.

As an application of the fixed-distance coupling we derive a maximum principle for the gradient of harmonic functions on manifolds with non-negative Ricci curvature.

As far as we are aware of, these constructions are new, though the existence of shy couplings on manifolds is suggested by Kendall in \cite{Kendall}.

\end{abstract}
\thanks{Both authors were partially supported by a grant of the Romanian National Authority for Scientific Research, CNCS - UEFISCDI, project number PN-II-RU-TE-2011-3-0259. The first author was also partially supported by a grant of the Romanian National Authority for Scientific Research, CNCS - UEFISCDI, project number PNII-ID-PCCE-2011-2-0015, and the second author was also partially supported by Marie Curie Action Grant PIRG.GA.2009.249200.}

%The first author was also partially supported by a grant of the Romanian National Authority for Scientific Research, CNCS - UEFISCDI, project number PNII-ID-PCCE-2011-2-0015,
%\date{}
%\subjclass[2010]{Primary 58J65, 53C44; Secondary 60H30}
%\keywords{Ricci flow, stochastic target problem, coupling}

\maketitle

\section{Introduction }

A first motivation of the present work is  the following (stochastic) modification
of the classical \emph{Lion and Man problem} of Rado (\cite{Littlewood}) on manifolds. Consider
a Brownian Lion $X_{t}$ and a Brownian Man $Y_{t}$ running on a $d$-dimensional Riemannian manifold $M$ (for instance the unit sphere in $\R^{3}$).

We describe two versions of the classical Lion and Man problem.
\begin{problem}[Fast/Finite Time Coupling]\label{p:1} Can the Lion capture the Man?

More precisely, given two distinct starting points $x,y\in M$ and a Brownian motion $Y_{t}$ on $M$ starting
at $y$, can one find a Brownian  motion $X_{t}$ on $M$ starting
at $x$ such that $\tau =\inf \left\{ t\geq 0:X_{t}=Y_{t}\right\} $ is almost
surely finite (or almost surely bounded)? A weaker version of this problem is whether for a
given $\epsilon >0$ and a given Brownian motion $Y_{t}$ on $M$ starting at $y$ one can find a
Brownian motion $X_{t}$ on $M$ starting at $x$ such that $
\tau =\inf \left\{ t\geq 0:d(X_{t},Y_{t}) =\epsilon
\right\} $ is almost surely finite (or almost surely bounded).  Here $d(x,y)$ stands for  the Riemannian distance on $M$.

\end{problem}

One example of coupling which is known in the literature as the \emph{mirror coupling}, and it was introduced by Lindvall and Rogers \cite{Lin-Rog} for processes defined on Euclidean spaces, and by Cranston in \cite{CranstonJFA} and Kendall \cite{Kendall5} in the case of processes defined on manifolds, the so-called \emph{Cranston-Kendall mirror coupling}.  It turns out that this coupling is a very useful and versatile construction when it comes to various geometric and analytic properties on manifolds.
For instance, it was shown in \cite{Kendall5}, for the case of manifolds with non-negative Ricci curvature, that the Man and the Lion must meet in finite time under this  mirror coupling.

Geometrically, the mirror coupling makes the motions $X_{t},Y_{t}$ move toward each other in the geodesic direction.   Closely related coupling is the \emph{synchronous} coupling in which the Brownian motions $X_{t},Y_{t}$ move parallel to each other in the geodesic direction and was used for example in \cite{MR2215664}.  On a different note, continuous versions of couplings of Brownian motions are constructed in \cite{MR2790368} and \cite[Theorem 10.37]{Stroock2}.

\emph{Though couplings under which the particles meet in finite time have received a lot of attention in the literature, as for instance the recent maximality properties analyzed in \cite{HsuSturm}, \cite{Kuw2} or \cite{Kuw1} it is not our interest in this paper.  }

If the couplings in Problem~\ref{p:1} are trying to meet as fast as possible, there is also the scenario of couplings which prevents the particles from meeting.  We formulate this as follows.

\begin{problem}[Strong Shy Coupling]\label{p:2}   Can the Man avoid being eaten by the Lion indefinitely?

More precisely, given two distinct starting points $x,y\in M$ and a Brownian motion $
X_{t}$  on $M$ starting at $x$,  can one find a Brownian motion $Y_{t}$ on $M$ starting at $y$ such
that almost surely $X_{t}\neq Y_{t}$ for all $t\geq 0$? A stronger version
of the question is whether the Brownian motion $Y_{t}$ can be chosen in such
a way that there exists an $\epsilon >0$ such that almost surely $%
d\left( X_{t},Y_{t}\right) \geq \varepsilon $ for all $t\geq 0$.

\end{problem}

The notion of \textit{shy coupling} of Brownian motions was introduced in
\cite{Burdzy-Benjamini} and subsequently studied in \cite{Burdzy-Kendal} and \cite{Kendall} and is a coupling for which, with positive probability,  the distance between the two
processes stays positive for all times.   A stronger version of shyness ($\epsilon$-shyness, $\epsilon>0$) asserts that with positive probability the distance between the processes is greater than $\epsilon$.  In this paper we use this latter version of shyness, in the stronger sense where the distance between the processes is greater than $\epsilon$ with probability $1$.

To set up the terminology, we mention that all couplings in the present paper are Markovian couplings in the sense of \cite{Burdzy-Benjamini} and introduced in Section~\ref{s:p}.

In a different direction, a synthetic notion of a lower bound on the Ricci curvature was settled in \cite{MR2480619,MR2237206,MR2237207} and is a very useful tool in analysis on measure metric spaces which is a very active area of research nowdays.  On the other hand, the notion of couplings and lower bound on Ricci curvature was pioneered in \cite{Kendall5}.  Related to this, a notion of Ricci curvature in discrete spaces appears in \cite{MR2484937} and see also \cite{MR2502429,MR3021521,MR2989449,MR2872958}.

In this spirit, a second motivation of our work comes from \cite[Corollary 1.4]{ReSt} which states the following.
\begin{corollary}\label{c:0}
On a complete Riemannian manifold $M$ the Ricci tensor satisfies $Ric\ge k$ if and only if there exits a conservative Markov process  $(\Omega, \mathcal{A},\p^{z},Z_{t})_{z\in M\times M,t\ge0}$ with values in $M\times M$ such that the coordinate processes $(X_{t})_{t\ge0}$ and $(Y_{t})_{t\ge0}$ are Brownian motions on $M$ and such that for all $z = (x,y)$ and all $t\ge0$,
\begin{equation}\label{e:i0}
d(X_{t},Y_{t})\le e^{-kt/2}d(x,y),\quad \p^{z}-a.s.
\end{equation}
\end{corollary}

The coupling that is used in \cite{ReSt} under the hypothesis that $Ric \ge k$ is the synchronous coupling alluded above.

A natural question, which fits our interests in the present paper, is to see if one can find couplings of Brownian motions $X_{t},Y_{t}$ such that \eqref{e:i0} is saturated.   For instance, if $k=0$ this amounts to finding a fixed-distance coupling which is in fact a strong version of a shy coupling.

Here is an outline of the paper.   Section~\ref{s:p} is about notations and basic results and notions.
In Section~\ref{s:6} we have the main result.   This states  that on a complete $d$-dimensional Riemannian manifold $M$ with positive injectivity radius, the Ricci curvature uniformly bounded from below and the sectional curvature uniformly bounded from above we can construct shy couplings.  This existence result of shy coupling on manifolds is also stated in Kendall \cite[Section 4]{Kendall} without proof but with a hint on how to do it.  Our approach is different.  Moreover, if the Ricci curvature is in addition non-negative, we can also construct fixed-distance couplings.  Finally, we show that if the Ricci curvature is actually bounded from below by a positive constant, then we can find fast approaching couplings, for which the distance between processes decays exponentially fast to $0$.  In fact our main result follows as a consequence of a much more general finding which shows that under some technical conditions on a function $F$ defined on an interval of the positive line, there exists a coupling of Brownian motions $X_{t},Y_{t}$ such that $\rho_{t}=d(X_{t},Y_{t})$ satisfies
\[
\frac{d\rho_{t}}{dt}=-\frac{1}{2}F(\rho_{t})
\]
for small times $t$.  Under the assumption that the Ricci curvature is non-negative, this can be extended for all values of $t\ge0$.  This is the content of Theorem~\ref{t:100} and it shows that there is much wider classes of couplings with deterministic distance.

Moreover, for a given function $\rho:[0,\infty)\to[0,\infty)$,  Corollary~\ref{Cor:9} gives conditions on $\rho$ such that this is realized as the distance function between two co-adapted Brownian motions.

We want to point a few details about the techniques.  In the first place we treat separately the cases when $d$ is odd, respectively even.  In the case of odd dimensional manifolds we can carry out the proof based on the orthonormal frame bundle.   For even dimensional manifolds we introduce the notion of $N$-frames at a point $x\in M$ which is  an embedding of the tangent space $T_{x}M$ into $\R^{N}$.  As it turns out, it suffices to use this construction for the particular case $N=d+1$, however, for the general $N$ this may be of independent interest by itself.   This is somewhat reminiscent of works on stochastic flows given for example in \cite{Elworthy,Elworthy2}.

Here is a brief exposition of the idea.  Suppose we have $X_{t}$ a  Brownian motions and want to exhibit another one $Y_{t}$ which is driven in some sense by $X_{t}$.   From a loose point of view what we do first  is to split the orthogonal to the tangent space at $X_{t}$ into orthogonal planes.  This splitting is possible only if the dimension $d$ is odd.  If this is the case, using the parallel transport along the geodesic, we can transport these planes at $X_{t}$ into orthogonal planes at $Y_{t}$.  Next we want the components of driving Euclidean Brownian motion at $X_{t}$ in these planes to be transported at $Y_{t}$ using parallel transport along the geodesic joining $X_{t}$ and $Y_{t}$ and then rotated by the same angle (chosen appropriately) in each of the transported planes at $Y_{t}$.   This is how we construct all three couplings first locally and then by patching them together to a global one.  In the even dimensional case using the $d+1$-frames we essentially add one more dimension to the tangent space and  carry out the same program.

In Section~\ref{s:10} we discuss some geometric aspects related to the main result in the previous section (Theorem~\ref{t:7}), and we present a localized version of the shy coupling, which is used in Section~\ref{s:11} to come back to the motivations of the paper, namely the Lion and the Man and also the connection with the lower bound on the Ricci curvature.

\section{Preliminaries}\label{s:p}

By $M$ we denote a Riemannian manifold.   In this paper all Riemannian manifolds are assumed to be complete.   For a given $d$-dimensional Riemannian manifold $M$, we use the standard notations from \cite{Elton} or \cite{Stroock2}
to denote by $\mathcal{O}(M)$ the orthonormal frame bundle.  For a given orthonormal frame $U$ at a point $x\in M$ and $\xi\in \R^{d}$, $H_{\xi}(U)$ is the horizontal lift of $U\xi\in T_{x}M$ at the point $U\in \mathcal{O}(M)$.   We will use the simpler notation of $H_{i}$ for $H_{e_{i}}$, with $\{e_{i}\}_{i=1,\dots,n}$ denoting the standard basis of $\R^{d}$.

We collect here some notions from differential geometry which will be used in the sequel.  The reader is referred to \cite{DoC} or \cite{Cheeger} for basic notions and results.    The curvature tensor $R_{x}$ at $x$ is $R_{x}(X,Y)=\nabla_{X}\nabla_{Y}-\nabla_{Y}\nabla_{X}-\nabla_{[X,Y]}$ and the Ricci tensor is the contraction $Ric_{x}(X,Y)=\sum_{i=1}^{d}\langle R_{x}(X,E_{i})E_{i},Y \rangle$, where $\{E_{i}\}_{i=1,\dots, d}$ is any orthonormal basis at $x$ and $X,Y\in T_{x}M$.   This definition of the Ricci tensor does not depend on the choice of orthonormal basis, and in the particular case of surfaces it simplifies to $Ric_{x}(X,Y)=K_{x}\langle X,Y\rangle$, where $K$ is the  Gauss curvature.

We denote by $d(x,y)$ the Riemannian distance between $x$ and $y$.

A geodesic on $M$ is a smooth curve $\gamma:[a,b]\rightarrow M$ such that $\ddot{\gamma}(s)=0$ for each $s\in[a,b]$, where the dot represents the covariant derivative along $\gamma$.  Throughout the paper we assume that the geodesics are running at unit speed.    For a point $x\in M$, we define $C_{x}$ to be the cutlocus of $x$, that is the set of points $y\in M$ for which the extension (beyond $x$ or $y$) of the minimizing geodesic between $x$ and $y$ ceases to be minimizing.  We will also use the notation $Cut\subset M\times M$, defined as the set of all points $(x,y)$ which are at each other's cut-locus. For points $x,y\in M$ which are not at each other's cut-locus, we define $\gamma_{x,y}$ to be the unique unit speed minimizing curve joining $x$ and $y$.

The injectivity radius is the smallest number $i(M)$ such that any point $x\in M$, the exponential map at $x$ is a diffeomorphism on the ball of radius $i(M)$ in the tangent space $T_{x}M$.

Given a geodesic $\gamma$, a Jacobi field along $\gamma$ is a vector field $J(s)$ such that
\begin{equation}\label{e0:1}
\ddot{J}(s)+R_{\gamma(s)}(J(s),\dot{\gamma}(s))\dot{\gamma}(s)=0,
\end{equation}
where the dot represents  the derivative along $\gamma$.

Given a vector field $V$ along a geodesic $\gamma$ defined on $[a,b]$, the index form $\mathcal{I}$ associated to it is defined as
\begin{equation}\label{e0:2}
\mathcal{I}(V,V)=\int_{a}^{b}(|\dot{V}(s)|^{2}-\langle R_{\gamma(s)}(V(s),\dot{\gamma}(s))\dot{\gamma}(s),V(s) \rangle) ds,
\end{equation}
and using polarization $\mathcal{I}$ can be extended to a bilinear form on the space of vector fields along the geodesic $\gamma$.  In the particular case when $J$ is a Jacobi field, an integration by parts formula shows that
\begin{equation}\label{e0:3}
\mathcal{I}(J,J)=\langle \dot{J}(b), J(b) \rangle-\langle \dot{J}(a),J(a) \rangle
\end{equation}
where $[a,b]$ is the definition interval of $\gamma$.

A manifold has constant curvature $r$ if the sectional curvature is $r$ for all choices of the two dimensional plane, that is $\langle R_{x}(X,Y)Y,X\rangle =r$ for any $x\in M$ and any orthogonal unit vectors $X,Y\in T_{x}M$.   In this case the Ricci curvature simplifies as well as the Jacobi field equation \eqref{e0:1}.   We record here the calculation, as it will be used later on.   Assume that $\gamma_{x,y}$ is the minimal geodesic between the points $x,y\in M$ which are not at each other's cut-locus, $\rho=d(x,y)$  and let $\xi\in T_{x}M$ and $\eta\in T_{y}M$ be two unit vectors.  Consider $\xi(s)$ the extension of $\xi$ by parallel transport along $\gamma$ from $x$ to $y$, and similarly let $\eta(s)$ be the extension of $\eta$ by parallel transport from $y$ to $x$.  The Jacobi field  $J_{\xi,\eta}$ whose value at $x$ is $\xi$ and $\eta$ at $y$ with $\xi$ and $\eta$ orthogonal to $\gamma$, can be computed as follows
\begin{equation}\label{e0:4}
J_{\xi,\eta}(s)=w_{1}(s)\xi(s)+ w_{2}(s)\eta(s)
\end{equation}
where $w_{1}$, $w_{2}$ solve the boundary value problems
\[
\begin{cases}
\ddot{w}_{1}+rw_{1}=0 \\
w_{1}(0)=1\\
w_{1}(\rho)=0
\end{cases}
\text{ and }\quad
\begin{cases}
\ddot{w}_{2}+rw_{2}=0 \\
w_{2}(0)=0\\
w_{2}(\rho)=1
\end{cases},
\]
whose solutions are
\begin{equation}\label{e0:5}
w_{1}(s)=\begin{cases} \frac{\sin(\sqrt{r}(\rho-s))}{\sin(\sqrt{r}\rho)} , & r\ne0 \\ \frac{\rho-s}{\rho}, & r=0\end{cases}\quad\text{ and }\quad w_{2}(s)=\begin{cases} \frac{\sin(\sqrt{r} s)}{\sin(\sqrt{r}\rho)},  & r\ne0 \\ \frac{s}{\rho}, & r=0 \end{cases}.
 \end{equation}

Next, we introduce the main notions regarding couplings.  Recall that in general by a coupling we understand a pair of processes $%
\left( X_{t},Y_{t}\right) $ defined on the same probability space, which are separately Markov, that is%
\[
\begin{split}
P\left( \left. X_{s+t}\in A\right\vert X_{s}=z,X_{u}:0\leq u\leq s\right)
&=P^{z}\left( X_{t}\in A\right) \\
P\left( \left. Y_{s+t}\in A\right\vert Y_{s}=z,Y_{u}:0\leq u\leq s\right)
&=P^{z}\left( Y_{t}\in A\right)
\end{split}
\]
for any measurable set $A$ in the state space of the processes.

The notion of \emph{Markovian coupling} as used in \cite{Burdzy-Benjamini} requires that in addition to the above, the joint process $(X_{t},Y_{t})$ is Markov and
\begin{equation}\label{eq:adapt}
\begin{split}
P\left( \left. X_{s+t}\in A\right\vert X_{s}=z,X_{u},Y_{u}:0\leq u\leq
s\right) &=P^{z}\left( X_{t}\in A\right) \\
P\left( \left. Y_{s+t}\in A\right\vert Y_{s}=z,X_{u},Y_{u}:0\leq u\leq
s\right) &=P^{z}\left( Y_{t}\in A\right)
\end{split}
\end{equation}
for any measurable set $A$ in the state space of the processes.

The notion of \emph{co-adapted coupling} (introduced by Kendall, \cite{Kendall}) is the same as the above but without the Markov property of $(X_{t},Y_{t})$.

The Markovian couplings are easily obtained for instance in the case when the process $(X_{t},Y_{t})$ is actually a diffusion on the manifold.  This would be the ideal case, but we still get a Markovian coupling if we patch together diffusion processes in a nice way.   For example this will be the case of the main construction on manifolds, where we start the coupling following a diffusion up to a certain stopping time, then, from the point it stopped we run it independently according to another diffusion and then stop this at another stopping time and so on.  We do this quietly without further details.

\section{Shy and Fixed-Distance Couplings on Riemannian Manifolds}\label{s:6}

In this section we prove a general result about the existence of shy coupling on Riemannian manifolds.  Before we launch into various technical details, we state the main result of this section.

\begin{theorem}\label{t:7}  Let $M$ be a complete $d$-dimensional Riemannian manifold, $d\ge2$, with positive injectivity radius and such that for some  real number $k$:
\begin{equation}\label{e7:0}
k\le Ric_{x}  \text{ for all } x\in M \text{ and }   \sup_{x\in M}K_{x}<\infty,
\end{equation}
where $Ric$ is the Ricci tensor and $K_{x}$ stands for the maximum of the sectional curvatures at $x\in M$.
\begin{enumerate}
\item For $k<0$, there exists $\epsilon,\delta>0$ such that for any points $x_{0},y_{0}\in M$ with $d(x_{0},y_{0})<\epsilon$ we can find a Markovian coupling of Brownian motions $X_{t},Y_{t}$ starting at $x_{0},y_{0}$ such that $d(X_{t},Y_{t})\ge d(x_{0},y_{0})$ for all $t\ge0$ and $d(X_{t},Y_{t})=e^{-kt/2}d(x_{0},y_{0})$ for $0\le t\le \delta$.

\item If $k\ge 0$, there exists $\epsilon>0$ such that for any $x_{0},y_{0}\in M$ with $d(x_{0},y_{0})<\epsilon$, there exists a Markovian coupling of Brownian motions $X_{t},Y_{t}$ starting at $x_{0},y_{0}$ such that
\[
d(X_{t},Y_{t})=e^{-kt/2}d(x_{0},y_{0})\text{ for all }t\ge0.
\]

\end{enumerate}
\end{theorem}

We will deduce this theorem as a particular case of the following more general result.

\begin{theorem}\label{t:100}  Assume the same geometric conditions as in Theorem~\ref{t:7}, namely, $d\ge2$, positive injectivity radius and \eqref{e7:0}.  Let $0<b\le \infty$ and $F:(0,b)\to\R$ be a smooth function such that for some $0\le a<d-1$,
\begin{equation}
-\frac{2a}{\rho}\le F(\rho) \le k\rho.
\end{equation}

\begin{enumerate}
\item There exist positive constants $\epsilon,\delta>0$ such that for any points $x_{0},y_{0}\in M$, with $d(x_{0},y_{0})\le \epsilon$, we can find a Markovian coupling of Brownian motions $X_{t},Y_{t}$ such that $X_{0}=x_{0}$, $Y_{0}=y_{0}$ and $\rho_{t}=d(X_{t},Y_{t})$ satisfies
\begin{equation}\label{e10:800}
\frac{d\rho_{t}}{dt}= -\frac{1}{2}F(\rho_{t}) \text{ with }\rho_{0}=d(x_{0},y_{0})
\end{equation}
for $t\in[0,\delta)$.

 \item Moreover, for $k<0$, we can actually take $\epsilon$ and $\delta$ to be small enough and extend this  coupling for all $t\ge0$ such that $d(X_{t},Y_{t})\ge\rho_{0}$.

\item In the case $k\ge0$, if in addition we have $0\le F(\rho)$, then we can find a small $\epsilon>0$ such that for any points $x_{0},y_{0}$ with $d(x_{0},y_{0})\le \epsilon$, there is a Markovian coupling of Brownian motions $X_{t},Y_{t}$ with $X_{0}=x_{0}$ and $Y_{0}=y_{0}$ such that $\rho_{t}=d(X_{t},Y_{t})$ satisfies \eqref{e10:800} for all $t\ge0$.
\end{enumerate}
\end{theorem}

What this theorem says is that we can obtain couplings where the distance function $\rho_{t}=d(X_{t},Y_{t})$ satisfies a prescribed differential equation in the form of \eqref{e10:800} (at least for short time). For instance, Theorem~\ref{t:7} is obtained simply for the case of $F(\rho)=k\rho$.

We point out that given $F$ as in the theorem, for $\rho_{0}$ small enough, there is a solution to \eqref{e10:800} for small time $t_{0}$.  In fact, one can actually estimate the time $t_{0}$ from the fact that $F(\rho)\ge-\frac{2a}{\rho}$, we obtain that $\rho^{2}_{t}\le \rho_{0}^{2}+2at$.  Therefore as long as $\rho_{0}$ is small enough and $t_{0}$ is also small enough, $\rho_{t}<b$ and thus the solution does not exit the domain of definition of $F$.  On the other hand, $F(\rho)\le k\rho$, gives that $\rho_{t}\ge\rho_{0}e^{-kt/2}$ for as long as the solution is defined, therefore, the solution $\rho_{t}$ does not hit $0$.  Therefore as long as the initial condition is small enough, say $\rho_{0}<b/4$ and $t<b^{2}/(8(a+1))$, the solution is well defined and it is also unique.

The plan of the proof is as follows.   First we set up an extension of the orthonormal frame bundle (which will be used in the case of even dimensional manifolds).  Then we define the equation of the coupling at the level of this frame bundle and we seek a local solution.   Once we show the local existence of the coupling, we use patching in order to prove the global existence of the coupling.

We split the proof into several subsections.

\subsection{$N$-frames and the associated bundle}

One of the constructions of the Brownian motion on a $d$-dimensional Riemannian manifold  uses the notion of orthonormal frame bundle.  We first  extend this notion by introducing the following.

\begin{definition}  Let $N\ge d$ be an integer number.
An $N$-frame $U$ in $T_{x}M$ is a map $U:\R^{N}\to T_{x}M$ such that $UU'=Id$.    Alternatively, $U$ is an $N$-frame at $T_{x}M$ if the map $U'$ is an isometric imbedding of $T_{x}M$ into $\R^{N}$.
\end{definition}

In this small subsection, to avoid confusion, we will use the notation of $\langle \cdot,\cdot\rangle_{T_{x}M}$ to denote the inner product in $T_{x}M$, while $\langle \cdot,\cdot\rangle_{\R^{N}}$ will denote the inner product in $\R^{N}$.

Abusing the language we often say that $U$ is an $N$-frame at $x$ rather than in $T_{x}M$.   Another way of describing $U$ is via the vectors $X_{i}=Ue_{i}$, $i=1\dots N$, where $e_{i}$ are the standard basis vectors in $\R^{N}$.  The condition that $U$ is an $N$-frame is actually equivalent to the condition that
\begin{equation}\label{e7:1}
\sum_{i=1}^{N}\langle \xi , X_{i}\rangle_{T_{x}M} X_{i} =\xi \text{ for all }\xi\in T_{x}M.
\end{equation}
Indeed, if $X_{i}=U e_{i}$, then for any $\xi\in T_{x}M$,  $\sum_{i=1}^{N}\langle \xi , X_{i}\rangle_{T_{x}M} X_{i} =U\sum_{i=1}^{N}\langle U'\xi , e_{i}\rangle_{\R^{N}} e_{i} =UU'\xi=\xi$.
Conversely, condition \eqref{e7:1} determines an $N$-frame $U:\R^{N}\to T_{x}M$ by prescribing
\[
U\eta=\sum_{i=1}^{N}\langle \eta , e_{i}\rangle_{\R^{N}} X_{i},
\]
noting that $U'\xi=\sum_{i=1}^{N}\langle \xi,X_{i} \rangle_{T_{x}M} e_{i}$,  which under \eqref{e7:1} gives $UU'=Id$, as needed.

Hence we have different characterizations of an $N$-frame,  as a projection, as an isometric embedding  and as a set of vectors $Ue_{i}$.

Given two points $x,y\in M$, an $N$-frame $\{X_{i}\}_{i=1}^{N}$ at $x$, and an isometry $A:T_{x}M\to T_{y}M$.  Then $\{AX_{i}\}_{i=1}^{N}$ is certainly an $N$-frame at $y$ because $
\sum_{i=1}^{N}\langle \xi , AX_{i}\rangle AX_{i} =A\sum_{i=1}^{N}\langle A'\xi , X_{i}\rangle X_{i}=AA'\xi=\xi.$

Also, it is easy to see that if $O$ is an orthogonal transformation of $\R^{N}$ and $U$ is an $N$-frame, then $UO$ is also an $N$-frame. As in the standard case of the orthonormal bundle, it is clear that $\mathcal{O}(M)$ is a smooth bundle over $M$ and $\pi :\mathcal{O}(M)\to M$ which assigns to each $N$-frame $U$ in $T_{x}M$ its base point $x$ (i.e. $\pi U=x$) is a smooth map.  In the terminology of differential geometry,  $\mathcal{O}(M)$ is actually a fiber bundle with the fiber being the Stiefel manifold $V_{d,N}$ constructed from the trivial principal bundle $M\times O(N)$ over $M$.

For each fixed $N$-frame $U$  at $x\in M$, the tangent space $T_{U}\mathcal{O}(M)$ splits into the horizontal part $T_{U}^{H}\mathcal{O}(M)$ obtained by lifting tangent vectors from $T_{x}M$ and the vertical part $T_{U}^{V}\mathcal{O}(M)$ which contains a special class of tangent vectors obtained by differentiating curves which are determined by the action of $O(N)$ in the fiber.  For references the reader can consult \cite{Elton} or  \cite{Stroock2} (the discussion there is intended for the orthonormal frame bundle, but nevertheless most of it extends naturally to this context).

Now, we define the fundamental vector fields $H_{i}$ on $\mathcal{O}(M)$ by the prescription that at each $U$, $(H_{i})_{U}$ is the lift of the vector $Ue_{i}$ from $T_{\pi U}M$.
The main property here is that the associated Bochner Laplacian
\[
\Delta_{B}=\sum_{i=1}^{N}H_{i}^{2}
\]
projects down onto $M$ as the Laplace operator.  The proof is as in \cite[Section 8.1.3]{Stroock2}, and for simplicity we just point out the main difference.   For a vector $\xi\in\R^{N}$, let $(H_{\xi})_{U}$ be the horizontal lift of $U\xi$ at $U$.   Then with the same proof as \cite[Equation 8.30]{Stroock2}, for any smooth function $f$ on $M$ we have
\[
(H_{\xi})_{U}\circ H_{\eta}(f\circ\pi)=\langle (\mathrm{Hess}f)_{\pi U}U\xi,  U\eta \rangle,
\]
where $\mathrm{Hess}f$ is the Hessian of $f$ on $M$.  Once this is established, we can continue with
\[
\begin{split}
\sum_{i=1}^{N}(H_{i})_{U}H_{i}(f\circ \pi)&=\sum_{i=1}^{N}\langle (\mathrm{Hess}f)_{\pi U}Ue_{i}, Ue_{i} \rangle=\sum_{i=1}^{N}\langle U'(\mathrm{Hess}f)_{\pi U}Ue_{i}, e_{i} \rangle \\
&=\tr(U'(\mathrm{Hess}f)_{\pi U}U)=\tr((\mathrm{Hess}f)_{\pi U}U\,U')=\tr((\mathrm{Hess}f)_{\pi U})\\ &=(\Delta_{M}f)(\pi U),
\end{split}
\]
where we used that the Laplacian on $M$ is simply the trace of the Hessian.   Thus
\begin{equation}\label{e7:2}
\pi_{*}\Delta_{B}=\Delta_{M}.
\end{equation}

Under the assumptions in \eqref{e7:0}, the Ricci curvature is bounded from below and from this we learn that  the Brownian motion on $M$ does not explode.  Thus the Brownian motion constructed on $\mathcal{O}(M)$ (more appropriately the solution to the martingale problem for $\Delta_{B}$) projects down into the Brownian motion on $M$ and exists for all times.

\subsection{The Coupling SDE}

Now we want to couple Brownian motions on $M$, and for this matter we consider couplings of the form described below.  Namely, for given points $x_{0}, y_{0} \in M$ and $N$-frames $U_{0}$ at $x_{0}$ and $V_{0}$ at $y_{0}$, consider the system
\begin{equation}\label{e:5b}
\begin{cases}
dU_{t}=\sum_{i=1}^{N}H_{i}(U_{t})\circ dW_{t}^{i}\\
dV_{t}=\sum_{i=1}^{N}H_{i}(V_{t})\circ dB_{t}^{i}\\
dB_{t}=O_{U_{t},V_{t}}dW_{t}\\
X_{t}=\pi U_{t}\\
Y_{t}=\pi V_{t}.
\end{cases}
\end{equation}
Here $W_{t}$ is an $N$-dimensional Brownian motion and $O_{U,V}$ is an orthogonal $N\times N$ matrix which depends smoothly on $U,V$, at least on a subset of $\mathcal{O}(M)\times\mathcal{O}(M)$ which will be specified later on.    This insures that $B_{t}$ is also an $N$-dimensional Brownian motion.   We do not impose additional conditions on the matrix $O_{U_{t},V_{t}}$ yet.

The same arguments as in \cite[Section 6.5]{Elton} show that the generator of the diffusion $(U_{t},V_{t})$ is given by
\[
\Delta^{c}=\Delta_{B,1}+\Delta_{B,2}+2\sum_{i=1}^{N}H_{e_{i}^{*},2}H_{i,1}
\]
where the subscript $1$ or $2$ represents the action with respect to the first or the second variable, and
$e_{i}^{*}=O_{U,V}e_{i}$.

Let $\rho_{t}=d(X_{t},Y_{t})$ be the distance between the processes $X_{t}$ and $Y_{t}$.  Also let  $\widetilde{d}(U,V)=d(\pi U,\pi V)$ be the lift of the distance function from $M$ into $\mathcal{O}(M)$.   Using It\^o's formula we have that
\begin{equation}\label{e7:3}
d\rho_{t}=\left(( H_{i,1}+H_{e^{*}_{i},2} )\widetilde{d}\right)(U_{t},V_{t})dW_{t}+\frac{1}{2}\left(\Delta^{c}\widetilde{d}\right)(U_{t},V_{t})dt,
\end{equation}
which is certainly valid in the region where $\pi U_{t}$ and $\pi V_{t}$ are not at each other's cut-locus.  Thus in order to have the distance function $\rho_{t}$ satisfy
\[
d\rho_{t}=-\frac{1}{2}F(\rho_{t})dt
\]
we need to cancel the martingale part, which is $\left(( H_{i,1}+H_{e^{*}_{i},2} )\widetilde{d}\right)(U_{t},V_{t})dW_{t}$ and also force the bounded variation part to be equal to $F(\rho_{t})dt$.

For the martingale part, notice that the first variation formula gives
\[
( H_{i,1}+H_{e^{*}_{i},2} )\widetilde{d}(U,V)=\langle VO_{U,V}e_{i},\dot{\gamma}_{X,Y}  \rangle_{\pi V}-\langle Ue_{i},\dot{\gamma}_{X,Y}  \rangle_{\pi U},
\]
where $X=\pi U$, $Y=\pi V$, and $\gamma_{X,Y}$ is the minimizing geodesic joining $X$ to $Y$, run at unit speed.  The bounded variation part comes from the second variation formula and produces
\begin{equation}\label{e7:5}
(\Delta^{c}\widetilde{d})(U,V)=\sum_{i=1}^{N}\mathcal{I}(J_{i},J_{i}),
\end{equation}
where $J_{i}$ is the Jacobi field along the geodesic joining $\pi U$ to $\pi V$, with values  $Ue_{i}$, $VO_{U,V}e_{i}$ at the endpoints.

In order to cancel the martingale part from \eqref{e7:3}, we need to impose the condition
\[
\langle Ue_{i},\dot{\gamma}_{X,Y}  \rangle_{\pi V}-\langle VO_{U,V}e_{i},\dot{\gamma}_{X,Y}  \rangle_{\pi U}=0,
\]
and for the bounded variation part, we need to have
\[
\sum_{i=1}^{N}\mathcal{I}(J_{i},J_{i})=-F(\tilde{\rho}).
\]

\subsection{Local Construction}  This part of the proof consists in showing that there exists $\eta>0$ sufficiently small such that for any $x,y\in M$ with $d(x,y)<\eta$ there is a smooth choice of $O_{U,V}$ on $\mathcal{N}_{\eta}(x,y)=\pi^{-1}(B(x,\eta))\times \pi^{-1}(B(y,\eta))$ for which
\begin{equation}\label{e7:15}
\langle Ue_{i},\dot{\gamma}_{\pi U,\pi V}  \rangle_{\pi U}-\langle VO_{U,V}e_{i},\dot{\gamma}_{\pi U,\pi V}  \rangle_{\pi V}=0 \text{ for }(U,V)\in \mathcal{N}_{\eta}(x,y)
\end{equation}
and
\begin{equation}\label{e7:16}
\sum_{i=1}^{N}\mathcal{I}(J_{i},J_{i})=-F(d(x,y)), \text{ for } (U,V)\in \mathcal{N}_{\eta}(x,y),
\end{equation}
where $J_{i}$ are the Jacobi fields with boundary values $Ue_{i}$ and $V O_{U,V}e_{i}$ at the endpoints of the minimizing geodesic joining $\pi U$ and $\pi V$. Note here that for small $\eta$, there is a unique minimizing geodesic joining $\pi U$ and $\pi V$, so everything is well defined in this case.

Take $\eta<i(M)/3$, where $i(M)$ is the  injectivity radius of $M$. In fact we are going to choose possibly smaller values of $\eta$ later in the construction, but for now assume that it is smaller than $i(M)/3$.

Now, assume that $x_{0},y_{0}\in M$ are two fixed starting points with distance $d(x_{0},y_{0})<\eta$.  We will construct the coupling $(U_{t},V_{t})$ in $\mathcal{N}_{\eta}(x_{0},y_{0})$.

 We can choose an orthonormal basis $E_{1},E_{2},\dots, E_{d}$ at $x$ such that $E_{1}=\dot{\gamma}_{x,y}(0)$ and such that each $E_{j}$ depends smoothly on $(x,y)\in B(x_{0},\eta)\times B(y_{0},\eta)$.  We can extend this basis $E_{1},\dots, E_{d}$ along $\gamma_{x,y}$ and continue to call it $E_{1},\dots, E_{d}$.  Now,  condition \eqref{e7:15} becomes
\begin{equation}\label{e7:31}
U'\dot{\gamma}_{x,y}=O_{U,V}'V'\dot{\gamma}_{x,y}.
\end{equation}
Next, let  us denote $J_{1,j}$ the Jacobi field along the minimizing geodesic joining $\pi U$ to $\pi V$ such that it equals $E_{j}$ at $\pi  U$ and $0$ at $\pi V$.  Similarly let $J_{2,j}$ be the Jacobi field which is $0$ at $\pi U$ and $E_{j}$ at $\pi V$.  Then, since
\[
J_{i}=\sum_{j=1}^{d}\langle Ue_{i},E_{j}\rangle J_{1,j}+\sum_{j=1}^{d}\langle VO_{U,V}e_{i},E_{j}\rangle J_{2,j}
\]
it follows that
 \begin{equation}\label{e7:30}
\sum_{i=1}^{N}\mathcal{I}(J_{i},J_{i})= \sum_{j=2}^{d}\mathcal{I}(J_{1,j},J_{1,j})+ \sum_{j=2}^{d}\mathcal{I}(J_{2,j},J_{2,j})+2\sum_{j,k=2}^{d}\langle O_{U,V}'V'E_{j},U'E_{k}\rangle\mathcal{I}(J_{1,j},J_{2,k}).
 \end{equation}
 The expression given by the last sum can be simplified as follows.   Let $\tau_{x,y}$ stand for the parallel transport map from $T_{x}M$ to $T_{y}M$ along the minimizing geodesic $\gamma_{x,y}$.   Consider the bilinear map $\Lambda_{x,y}:T_{x}M\times T_{x}M\to \R$ defined by
 \[
 \Lambda_{x,y}(\xi,\eta)=\mathcal{I}(J_{1,\xi},J_{2,\eta}),
 \]
 where $J_{1,\xi}$ is the Jacobi field along $\gamma_{x,y}$ which is $\xi$ at $x$ and $0$ at $y$, and $J_{2,\eta}$ is $0$ at $x$ and $\tau_{x,y}\eta$ at $y$.  Another way of looking at this is as a linear map from $T_{x}M$ into itself, map which we still call $\Lambda_{x,y}$.    We can see this map also as a linear transformation preserving the orthogonal to $\dot{\gamma}_{x,y}$ at $x$ and we will denote this restriction also by $\Lambda_{x,y}$.    In fact, the actions of $\Lambda_{x,y}$ and its transpose on $\dot{\gamma}_{x,y}$ are zero.

 With this notation, it is not hard to see that for $N$-frames $U$ and $V$ at $x$, respectively at $y$, we have
 \begin{equation}\label{e7b:3}
 \sum_{j,k=2}^{d}\langle O_{U,V}'V'E_{j},U'E_{k}\rangle\mathcal{I}(J_{1,j},J_{2,k})=\tr(U O_{U,V}'V'\tau_{x,y} \Lambda_{x,y}).
 \end{equation}

 For the first part of the theorem we want to find a map $O_{U,V}$ such that \eqref{e7:15} is satisfied which is equivalent to
 \begin{equation}\label{e7:100:2}
 U'\dot{\gamma}_{x,y}=O_{U,V}'V'\dot{\gamma}_{x,y}
 \end{equation}
 In addition we want to fulfill \eqref{e7:16} which is the same as asking that quantity in \eqref{e7:30} equals $-F(d(x,y))$.  Therefore equation \eqref{e7:16} becomes in this reformulation
 \begin{equation}\label{e7:100}
  \tr(U O_{U,V}'V'\tau_{x,y} \Lambda_{x,y})=-\frac{1}{2}\left(\sum_{j=2}^{d}\mathcal{I}(J_{1,j},J_{1,j})+ \sum_{j=2}^{d}\mathcal{I}(J_{2,j},J_{2,j}) +F(\rho) \right)
 \end{equation}
where for simplicity of notations, we are going to denote $d(x,y)=\rho$.

 To carry this task through, we are going to use the following standard comparison result, whose proof can be found for instance in \cite[pp. 216-217]{DoC}.

\begin{lemma}\label{l:2}

Assume that $M$ and $\widetilde{M}$ are two manifolds and $\gamma$, $\widetilde{\gamma}$ are two normalized geodesics defined on $[0,\rho]$ such that $\widetilde{\gamma}$ does not have conjugate points.   Assume that $J_{t}$ and $\widetilde{J}_{t}$ are two Jacobi vector fields along $\gamma$, respectively $\widetilde{\gamma}$,  such that $J_{0}=\widetilde{J}_{0}=0$, $|J_{\rho}|=|\widetilde{J}_{\rho}|$, $\langle \dot{J}_{0},\dot{\gamma}(0)\rangle= \langle \dot{\tilde{J}}_{0},\dot{\tilde{\gamma}}(0)\rangle$ and
\[
K^{+}(\gamma(t))\le\widetilde{K}^{-}(\widetilde{\gamma}(t)),
\]
where $K^{+}(x)$ is the maximum of the sectional curvature at $x$ and $\widetilde{K}^{-}(\widetilde{x})$ is the minimum of the sectional curvature at $\widetilde{x}$. Then we have
\begin{equation}\label{e7:9}
\mathcal{I}(\widetilde{J},\widetilde{J})\le \mathcal{I}(J,J).
\end{equation}

\end{lemma}

Since the sectional curvature is bounded from above, $K_{x}\le1/\omega^{2}$ for all $x\in M$ for a small enough $\omega>0$.   With this choice, for points $x,y\in M$ at distance $\rho=d(x,y)<\pi \omega/4$, comparing the index form of the manifold $M$ with the index form of a sphere of radius $r$,  for geodesics of length $\rho<\pi \omega/4$, we obtain
\[
\mathcal{I}(\widetilde{J},\widetilde{J})\le \mathcal{I}(J,J),
\]
where $J,\widetilde{J}$ are as in the Lemma~\ref{l:2}.  On the other hand, for the $d$-dimensional sphere $S^{d}$ we have $\widetilde{J}(s)=w_{2}(s)\widetilde{E}(s)$, where $w_{2}$ is given by \eqref{e0:4} and $\widetilde{E}$ is the parallel transport of $\widetilde{E}_{0}\in T_{\widetilde{\gamma}(0)}S^{d}$ along $\widetilde{\gamma}$. From \eqref{e0:3} and \eqref{e0:5} (notice that the $r$ there is the curvature bound which in our case at hand is $1/\omega^{2}$) we conclude that
\[
\mathcal{I}(\widetilde{J},\widetilde{J})= \dot{w}_{2}(\rho)=\frac{\cot(\rho/\omega)}{\omega}
\]
and consequently, we obtain
\begin{equation}\label{e7b:21}
0<\frac{\cot(\rho/\omega)}{\omega}=\mathcal{I}(\widetilde{J},\widetilde{J})\le \mathcal{I}(J,J).
\end{equation}

We now choose $\eta$ sufficiently small, for instance smaller than $\omega$ above and also less than a third of the injectivity radius of $M$.

Recall that we want to choose $O_{U,V}$ so that  \eqref{e7:100:2} and \eqref{e7:100} are satisfied.

To show this, we recall another standard result in Riemannian geometry as for instance appears in \cite[Corollary 8.10]{Spivak}.

\begin{lemma}\label{l:1}  Assume $\gamma$ is a normalized geodesic on $[0,\rho]$ without conjugate points on it.  If $J$ and $V$ are two vector fields with the same boundary values, and $J$ is also a Jacobi field, then
\begin{equation}\label{e7:8}
\mathcal{I}(J,J)\le \mathcal{I}(V,V).
\end{equation}
\end{lemma}

Next we have the obvious equality
\[
 \sum_{j=2}^{d}\mathcal{I}(J_{1,j},J_{1,j})+ \sum_{j=2}^{d}\mathcal{I}(J_{2,j},J_{2,j})+2\sum_{j=2}^{d}\mathcal{I}(J_{1,j},J_{2,j}) = \sum_{j=2}^{d}\mathcal{I}(J_{1,j}+J_{2,j},J_{1,j}+J_{2,j}).
 \]
On the other hand, using the above comparison theorem with the vectors $E_{j}$ in place of $V$ and $J_{1,j}+J_{2,j}$ as the Jacobi field $J$, we obtain
\begin{equation}\label{et:1}
\begin{split}
\sum_{j=2}^{d}\mathcal{I}(J_{1,j}+J_{2,j},J_{1,j}+J_{2,j})\le \sum_{j=2}^{d}\mathcal{I}(E_{j},E_{j})&=\sum_{j=2}^{d}\int_{0}^{\rho}\left( |\dot{E}_{j}(s)|^{2}-\langle R(\dot{\gamma}(s),E_{j}(s))E_{j}(s),\dot{\gamma}(s)\rangle \right)ds \\
& =-\int_{0}^{\rho} Ric_{\gamma(s)}(\dot{\gamma}(s),\dot{\gamma}(s))ds\le -k\rho
\end{split}
\end{equation}
where $\rho=d(x,y)$, and therefore
\begin{equation}\label{e7b:20}
2\sum_{j=2}^{d}\mathcal{I}(J_{1,j},J_{2,j})\le -\left( \sum_{j=2}^{d}\mathcal{I}(J_{1,j},J_{1,j})+ \sum_{j=2}^{d}\mathcal{I}(J_{2,j},J_{2,j})+k\rho\right).
\end{equation}

 In the basis $E_{1}=\dot{\gamma}_{x,y}, E_{2},\ldots,E_{d}$ we can take
 \begin{equation}\label{et:2}
 f_{j}=U'E_{j} \quad\text{ and } \quad h_{j}=V'E_{j}, \qquad j=1,\ldots,d.
 \end{equation}

To choose the matrix $O_{U,V}$ as in \eqref{e7:100} we treat separately the cases of odd and even dimensional manifolds, as follows.

\emph{Case I:  $d$ is odd}.   In this case we take $N=d$, so we are back to the classical situation of the orthonormal frame bundle.   Let $A_{U}$ and $A_V$  be the (unique) orthogonal matrices which send $e_{j}$ into $f_{j}$, respectively $e_{j}$ into $h_{j}$, $j=1,\ldots,d$.   We set
\begin{equation}\label{eq:D}
\Delta_{x,y}=A_{V}'V'\tau_{x,y} \Lambda_{x,y}UA_{U}.
\end{equation}

We will choose the matrix $O_{U,V}$ such that, in addition to \eqref{e7:100} we also have
\[
A_{U}'O_{U,V}'A_{V}e_{1}=e_{1}.
\]
This is done as follows.  We will construct an orthogonal matrix $B_{x,y}$ such that
\begin{equation}\label{e7:200}
B_{x,y}e_{1}=e_{1} \text{ and }\tr(B_{x,y}\Delta_{x,y})=-\frac{1}{2}\left(\sum_{j=2}^{d}\mathcal{I}(J_{1,j},J_{1,j})+ \sum_{j=2}^{d}\mathcal{I}(J_{2,j},J_{2,j}) +F(\rho) \right).
\end{equation}
Once this is done, we can take
\[
O_{U,V}=A_{V}B_{x,y}'A_{U}',
\]
which then shows that \eqref{e7:100} and consequently \eqref{e7:16} are satisfied.

To get to terms with  $B_{x,y}$, we choose it to be given in matrix form by
\begin{equation}\label{e7b:10}
B_{x,y}=\met{1 & 0 &0 & 0 & 0& 0 &0 & 0\\  0 &\cos \alpha & \sin \alpha & 0 & 0 & 0 &0 &0  \\ 0 & -\sin \alpha & \cos \alpha & 0&0 &0 &0 &0 \\ 0 & 0 &0 &\cos \alpha & \sin\alpha & 0 &0 &0 \\  0  & 0 &0&-\sin \alpha & \cos\alpha & 0&0&0 \\ \hdotsfor{8} \\ \hdotsfor{8} \\ 0 &0&0 &0 &0 &0 &\cos\alpha &\sin\alpha \\ 0&0&0&0&0&0&-\sin\alpha&\cos\alpha}.
\end{equation}
This is where we actually use the fact that the dimension $d$ is odd: in the above representation we use on the diagonal $(d-1)/2$ blocks of $2\times 2$ unitary matrices.   With this choice, we clearly have  $B_{x,y}e_{1}=e_{1}$ and also $\Delta_{x,y}e_{1}=0$.  Furthermore,  because $B_{x,y}e_{2i}=\cos(\alpha)e_{2i}-\sin(\alpha)e_{2i+1}$ and  $B_{x,y}e_{2i+1}=\cos(\alpha)e_{2i+1}+\sin(\alpha)e_{2i}$, we get that
\[
\tr(B_{x,y}\Delta_{x,y})=\sum_{i=1}^{(d-1)/2}(\langle \Delta_{x,y}B_{x,y}e_{2i},e_{2i} \rangle+\langle \Delta_{x,y}B_{x,y}e_{2i+1},e_{2i+1} \rangle)=\cos(\alpha)\tr\left( \Delta_{x,y}\right)+\sin (\alpha) F_{x,y}
\]
with $F_{x,y}=\sum_{i=1}^{(d-1)/2}\left(\langle\Delta_{x,y}e_{2i+1},e_{2i}\rangle-\langle\Delta_{x,y}e_{2i},e_{2i+1}\rangle \right)$.  Hence, \eqref{e7:200} becomes equivalent to finding $\alpha\in[0,2\pi]$ such that
\begin{equation}\label{e7:100:b}
2\cos(\alpha)\tr\left( \Delta_{x,y}\right)+2\sin (\alpha) F_{x,y}=-\Bigg( \sum_{j=2}^{d}\mathcal{I}(J_{1,j},J_{1,j})+ \sum_{j=2}^{d}\mathcal{I}(J_{2,j},J_{2,j})+F(\rho)\Bigg).
\end{equation}
The key point now is that \eqref{e7b:20} is nothing but the statement that
\[
2\tr\left(\Delta_{x,y} \right)\le -\Bigg( \sum_{j=2}^{d}\mathcal{I}(J_{1,j},J_{1,j})+ \sum_{j=2}^{d}\mathcal{I}(J_{2,j},J_{2,j})+k\rho\Bigg).
\]
In addition to this, since $F(\rho)\ge -2a/\rho$ for small $\rho$ with $a<d-1$, combined with inequality \eqref{e7b:21} gives that,
\[
-\Bigg( \sum_{j=2}^{d}\mathcal{I}(J_{1,j},J_{1,j})+ \sum_{j=2}^{d}\mathcal{I}(J_{2,j},J_{2,j})+F(\rho)\Bigg)<-2(d-1)\frac{\cot(\rho/\omega)}{\omega}-F(\rho)\le-2(d-1)\frac{\cot(\rho/\omega)}{\omega}+2a/\rho<0
\]
for small enough $\rho$ (in fact, it suffices to take small $\rho/\omega$).

On the other hand, since $F(\rho)\le k\rho$ for small $\rho$, we have that
\[
2\tr\left(\Delta_{x,y} \right)\le -\Bigg( \sum_{j=2}^{d}\mathcal{I}(J_{1,j},J_{1,j})+ \sum_{j=2}^{d}\mathcal{I}(J_{2,j},J_{2,j})+k\rho\Bigg)\le -\Bigg( \sum_{j=2}^{d}\mathcal{I}(J_{1,j},J_{1,j})+ \sum_{j=2}^{d}\mathcal{I}(J_{2,j},J_{2,j})+F(\rho)\Bigg).
\]

We have now come to the key point of the construction of $O_{U,V}$, namely solving equation \eqref{e7:100}.  After all these preliminaries, \eqref{e7:100} is in fact equivalent to showing that there exists an angle $\alpha$ such that  \eqref{e7:100:b} is satisfied.  Finally,  simple trigonometry shows that for any $a<c<0$ and any $b$, the equation
\[
\cos(\alpha)a+\sin(\alpha)b=c
\]
has one solution as
\[
\sin(\alpha)=\frac{bc-a\sqrt{a^{2}-c^{2}+b^{2}}}{a^{2}+b^{2}} \text{ and }\cos(\alpha)=\frac{ac+b\sqrt{a^{2}-c^{2}+b^{2}}}{a^{2}+b^{2}}
\]
Taking now $a=2\tr(\Delta_{x,y})$, $b=2F_{x,y}$ and $c=-\Bigg( \sum_{j=2}^{d}\mathcal{I}(J_{1,j},J_{1,j})+ \sum_{j=2}^{d}\mathcal{I}(J_{2,j},J_{2,j})+F(\rho)\Bigg)$ shows that \eqref{e7:100:b} has a solution, in conclusion \eqref{e7:100} does too.  In particular,  the matrix $B_{x,y}$ depends smoothly on $U$ and $V$, hence $O_{U,V}$ also depends smoothly on $U$ and $V$.

\emph{Case II:  $d$ is even.}  In this case we use $N=d+1$.   Recall that we use $e_{1},e_{2},\dots, e_{d+1}$ to denote the standard basis of $\R^{d+1}$ and the vectors $f_{j}$, respectively $h_{j}$ are defined in \eqref{et:2}.    Furthermore, we have a set of $d$ orthogonal vectors, $f_{1},f_{2},\dots, f_{d}$ in a $d+1$ dimensional space.  We then define
\[
f_{d+1}=f_{1}\wedge f_{2}\wedge \dots \wedge f_{d}.
\]
to be the exterior product of the previous $d$ vectors.  With this addition, the vectors $f_{1},f_{2},\dots, f_{d+1}$ form an orthonormal basis in $\R^{d+1}$.  We do the similar thing to the vectors $h_{1},h_{2},\dots, h_{d}$ by defining $h_{d+1}$ to be the exterior product of $h_{1},h_{2},\dots, h_{d}$.

The difference from the previous case is that this time we consider the matrix $A_{U}$ which sends $e_{j}$ into $f_{j}$, $j=1,\ldots, d$, and the vector $e_{d+1}$ into $f_{d+1}$.  Clearly with this choice, $A_{U}$ is actually an orthogonal matrix in $\R^{d+1}$.   Similarly we define the matrix $A_{V}$ to be the matrix sending $e_{i}$ into $h_{i}$ for $i=1,2,\dots,d$ and $e_{d+1}$ into $h_{d+1}$.  Again, $A_{V}$ is an orthogonal matrix.

The rest of the argument is now the same argument as in the case when $d$ is odd, with the choice of $B_{x,y}$ as a $(d+1)\times (d+1)$ matrix such as the one in \eqref{e7b:10} and $\Delta_x,y$ as in \eqref{eq:D} Notice the catch here, namely the dimension of the matrix is $d+1$, an odd number!  The rest of the argument runs exactly in the same way as above with the obvious adjustments.  For instance, equation \eqref{e7:100:b} is the same, only that this time
\[
F_{x,y}=\sum_{i=1}^{d/2}\left(\langle\Delta_{x,y}e_{2i+1},e_{2i}\rangle-\langle\Delta_{x,y}e_{2i},e_{2i+1}\rangle \right)
\]
and the rest of the proof follows the same steps.

Let's wrap up the main findings of this subsection.   We showed that there exists (again, for small $\eta$) a matrix $O_{U,V}$ which depends smoothly on $(U,V)\in\mathcal{N}_{\eta}(x_{0},y_{0})$ such that \eqref{e7:15} and \eqref{e7:16} are satisfied.   In fact we proved that for small enough $\eta>0$, as long as the distance between $x_{0}$ and $y_{0}$ is less than $\eta/2$ and the process  $\left(X_{t},Y_{t}\right)$ stays inside $B(x_{0},\eta)\times B(y_{0},\eta)$, the distance function satisfies $\rho_{t}=\nu_{t}$ (the solution to \eqref{e10:800}).

\subsection{ The construction of the coupling  }\label{extension}  Consider first two independent $N$-dimensional Brownian motions $W_{t}$ and $\widetilde{W}_{t}$.   For a given stopping time $\tau$, we denote $W_{t,\tau}=W_{t}-W_{\tau}$.

We have proved that for a small enough $\eta>0$ and  any $x,y$ with $d(x,y)<\eta$ there exists a smooth choice $O_{U,V}$ on $\mathcal{N}_{\eta}(x,y)$. We will now use this to give a construction of the coupling as indicated in the statement of the theorem.

For any $\eta>0$ we define the $\eta$-neighborhood of the diagonal in $M\times M$ by
\[
D_{\eta}=\{ (x,y): d(x,y)\le\eta \},
\]
and let us also set
\[
\mathcal{D}_{\eta}=\{ (U,V)\in\mathcal{O}(M)\times\mathcal{O}(M): (\pi U,\pi V)\in D_{\eta} \}.
\]

For a fixed pair of points $(x_{0},y_{0})\in D_{\eta/4}$ and frames $U_{0}, V_{0}$ at $x_{0}$, respectively at $y_{0}$, we consider an orthonormal basis $E_{1},\dots E_{d}$ at $x_{0}$ with $E_{1}=\dot{\gamma}_{x_{0},y_{0}}(0)$ and extend this to a local orthonormal basis on $B(x_{0},2\eta)$ and then by parallel transport also to $B(y_{0},2\eta)$.   Using the local recipe outlined above we can construct a coupling with $\rho_{t}=\nu_{t}$ up to the first time $t$ when the base process $(X_{t},Y_{t})$ hits the boundary of the set $B(x_{0},\eta)\times B(y_{0},\eta)$.   Let's call this exit time $\tau_{1}$.  At $(x_{1},y_{1})=(X_{\tau_{1}},Y_{\tau_{1}})$ we have the orthogonal basis  $E_{1},\dots,E_{d}$ used in the local construction, which at $x_{1}$ satisfies  $E_{1}=\dot{\gamma}_{x_{1},y_{1}}$, and $U_{1}:=U_{\tau_{1}}$ and $V_{1}:=V_{\tau_{1}}$ are the frames obtained from \eqref{e:5b}.

The next step is to extend the construction of the coupling beyond time $\tau_{1}$.   There are two cases to be considered here.

If the point $(x_{1},y_{1})$ lies inside $D_{\eta/2}$, we can use the starting point $(x_{1},y_{1})$ and  continue to run $(U_{t},V_{t})$ following \eqref{e:5b} using now the Brownian motion $W_{t,\tau_{1}}$ with the time range $t\ge\tau_{1}$.  As above we let $\tau_{2}$ be the first time the process $(X_{t+\tau_{1}},Y_{t+\tau_{1}})$ hits the boundary of $B(x_{1},\eta)\times B(y_{1},\eta)$, and we set $(x_{2},y_{2})=(X_{\tau_{1}+\tau_{2}},Y_{\tau_{1}+\tau_{2}})$ and also $U_{2}=U_{\tau_{1}+\tau_{2}}$ and $V_{2}=V_{\tau_{1}+\tau_{2}}$.

On the other hand, if the point $(x_{1},y_{1})$ lands outside $D_{\eta/2}$, then we run the motions $U_{t}$ and $V_{t}$ for  $t\ge\tau_{1}$ with the system
\[
\begin{cases}
dU_{t}=\sum_{i=1}^{N}H_{i}(U_{t})\circ dW_{t,\tau_{1}}^{i}\\
dV_{t}=\sum_{i=1}^{N}H_{i}(V_{t})\circ d\widetilde{W}_{t,\tau_{1}}^{i}\\
X_{t}=\pi U_{t}\\
Y_{t}=\pi V_{t}.
\end{cases}
\]

In other words, $U_{t},V_{t}$ run as independent Brownian motions on $\mathcal{O}(M)\times\mathcal{O}(M)$, and $X_{t},Y_{t}$ run as independent Brownian motions on the base manifold $M$.  We continue with this construction for time $t$ in the interval $[\tau_{1},\tau_{1}+\tau_{2}]$, where the terminal time $\tau_{1}+\tau_{2}$ is the first time the process $(X_{t},Y_{t})$ lands in $D_{\eta/4}$, and we denote $(x_{2},y_{2})=(X_{\tau_{1}+\tau_{2}},Y_{\tau_{1}+\tau_{2}})$.

In both cases above we constructed the processes $U_{t},V_{t}$ defined up to the time $\tau_{1}+\tau_{2}$, and $(x_{2},y_{2})$ is either in $D_{\eta/2}$ or outside it.  Inductively, we can now repeat the construction above, to show that we can extend the construction of the processes for another $\tau_{3}$ units of time, and so on.   If for a certain $n$, $\tau_{n}=+\infty$, then we certainly take all other stopping times $\tau_{m}=0$ for $m>n$.

One of the main problems is to show that the construction can be extended for all times $t\ge0$, in other words that
\[
\sum_{n\ge1}\tau_{n}=+\infty.
\]
We are going to do this separately for the first part of the theorem, and argue differently for the second and third part.

For the case $k<0$, the idea is that as long as the process $(X_{t},Y_{t})$ stays inside $D_{\eta/2}$, we know that the distance process $\rho_{t}$ satisfies
\[
\frac{d\rho_{t}}{dt}=-\frac{1}{2}F(\rho_{t}),
\]
thus $\rho_{t}'\ge -k\rho_{t}/2$ which implies that $\rho_{t}$ is actually increasing as a function of $t$. This means that if $\eta$ is small enough, then in finite (deterministic) time, the process $(X_{t},Y_{t})$ exits $D_{\eta/2}$.  Once the process $(X_{t},Y_{t})$ exits the set $D_{\eta/2}$, $X_t$ and $Y_t$ run independently until they hit the set $D_{\eta/4}$, and then they stay in $D_{\eta/2}$ for at most a finite (deterministic) amount of time, after which they exit again $D_{\eta/2}$.  In particular we see that the processes $X_{t},Y_{t}$ have to run independently infinitely many times, and it is this fact that allows us to show that $\sum_{n\ge1}\tau_{n}=+\infty$.  This is done using the Borel-Cantelli's lemma.

For the moment, assume that we have two independent Brownian motions $X_{t},Y_{t}$ starting at $x_{0},y_{0}$ with $d(x_{0},y_{0})=\eta/2$.  If $\tau$ is the first time when $X_{t},Y_{t}$ are within distance $\eta/4$ to each other, we want to get an estimate on $\p(\tau>\delta)$ for some $\delta>0$.   To do this, we use the following inclusion
\[
\{\zeta_{X,\eta/16}>\delta \}\cap\{\zeta_{Y,\eta/16}>\delta \}\subset\{ \tau>\delta\}
\]
where $\zeta_{X,\eta/16}$ is the first exit time of $X_{t}$ from the ball $B(x_{0},\eta/16)$ and similarly $\zeta_{Y,\eta/16}$ is the first time  $Y_{t}$ exits the ball $B(y_{0},\eta/16)$.   This inclusion can be stated in words as follows.   If $X_t$ and $Y_t$ stay inside $B(x_{0},\eta/16)$, respectively $B(y_{0},\eta/16)$, up to time $\delta$, and since $x_{0},y_{0}$ are distance $\eta/2$ apart, it follows that $X_t$ and $Y_t$ are not within $\eta/4$ of each other in the time interval $[0,\delta]$.  The conclusion we draw from this is that
\[
\p(\tau>\delta)\ge\p(\zeta_{X,\eta/16}>\delta)\p(\zeta_{Y,\eta/16}>\delta).
\]
Finally, since the the Ricci curvature is bounded below, we can invoke now the estimate on the exit times from balls, for instance \cite[Theorem 3.6.1]{Elton}, to obtain that for any point $x$ on $M$ we have
\[
\p_{x}(\zeta_{\eta/16}\le \delta)\le e^{-Cr^{2}/\delta},
\]
where the constant $C>0$ depends only on the lower bound on the Ricci curvature and the dimension of the manifold. Thus for a fixed $\eta>0$ we obtain that
\begin{equation}\label{e7:40}
\p_{x}(\zeta_{\eta/16}> \delta)>1- e^{-C\eta^{2}/\delta}:=C_{2}>0,
\end{equation}
for a certain constant $C>0$, and therefore
\[
\p(\tau>\delta)\ge C_{2}^{2}.
\]

With this at hand we get that
\[
\sum_{n\ge1}\p(\tau_{n}>\delta)=+\infty,
\]
and using Borel-Cantelli's lemma we conclude that $\sum_{n\ge1}\tau_{n}=+\infty$, which shows that the construction of the coupling extends for all times $t\geq 0$.

For the other case of $k\ge0$ and $F(\rho)\ge0$, clearly $\nu_{t}$ is going to be non-increasing and the bulk of the argument is complementary to the previous one.
More precisely, in the above proof it was the independent motions which played the main role, while here the main role is played by the coupling.   To get to terms, note  that if we start the coupling with points $x_{0},y_{0}$ such that $d(x_{0},y_{0})<\eta/4$, then, since the distance between the processes does not increase, the process $(X_{t},Y_{t})$ stays in $D_{\eta/2}$ up to the time $\sum_{n\ge1}\tau_{n}$.   The issue is to show that this sum is always infinite.  What we want to do is to find a lower bound on $\p(\tau_{1}>\delta)$.   Using the same notation as above, we have
\begin{equation}\label{*}%\tag{*}
\{ \zeta_{X,\eta/16} >\delta \}\subset \{\tau_{1}>\delta \}.
\end{equation}
To see this, we follow the construction until either $X$ or $Y$ hit the ball of radius $\eta$ centered at $x_{0}$, respectively $y_{0}$.   Now,  if $X$ stays inside $B(x_{0},\eta/16)$ on the time interval $[0,\delta]$, since $d(x_{0},y_{0})<\eta/4$ and the processes remain at fixed or non-increasing distance, an application of the triangle inequality shows that $Y$ remains inside $B(y_{0},9\eta/16)$ on the time interval $[0,\delta]$, which in turn implies \eqref{*}.   Using again \eqref{e7:40} we get that
\[
\p(\tau_{1}>\delta)\ge C_{3}>0
\]
for a constant $C_{3}$ which is independent of the starting points.   Since this is applicable to all stopping times $\tau_{n}$, we learn again from Borel-Cantelli's lemma that $\sum_{n\ge1}\tau_{n}=+\infty$.

\subsection{Finishing off}  In the previous section we constructed the coupling and we proved that it is defined for all times.  We now want to show that the construction actually does what the Theorem asks for.  This is already spelled out in the previous subsection in a certain form.

For the first part ($k<0$), on each of the regions where the coupling is inside $D_{\eta/2}$, the distance is non-decreasing, and therefore it is larger than the starting distance which is at most $\eta/4$.  On the other hand, if the coupling exits $D_{\eta/2}$, then it runs as independent Brownian motions until it hits again $D_{\eta/4}$, and consequently the distance is at least $\eta/4$ apart.  In both regimes the distance does not get smaller than the starting distance and this concludes the proof of the second part of Theorem~\ref{t:100}.

For the last part of the Theorem,  the coupling never leaves $D_{\eta/2}$ and for all times the distance functions $\rho_{t}$ equals the solution of the equation \eqref{e10:800}.   \qedhere

Though we are done proving the Theorem, we put here an interesting consequence of the proof.  There is a more general statement which guarantees the existence of a coupling which is not necessarily Markovian but co-adapted and its proof is based on a very simple modification of the proof which will leave to the reader.
\begin{corollary}\label{Cor:9}
Assume the same geometric assumptions as in Theorem~\ref{t:7} ($d\ge2$, positive injectivity radius and \eqref{e7:0}).

Let $T>0$ and $\rho:[0,T)\to[0,\infty)$ be a function such that for some $0\le a<d-1$, we have
\begin{equation}\label{e:n:c:1}
-\frac{\rho(t)}{2}\le \rho'(t)\le \frac{a}{\rho(t)} \text{ with }\rho(0)=\rho_{0}.
\end{equation}
\begin{enumerate}
\item There exist positive constants $\epsilon,\delta>0$ such that for any points $x_{0},y_{0}\in M$, with $d(x_{0},y_{0})\le \epsilon$, we can find a co-adapted coupling of Brownian motions $X_{t},Y_{t}$ such that $X_{0}=x_{0}$, $Y_{0}=y_{0}$ and $d(X_{t},Y_{t})=\rho(t)$ for $t\in[0,\delta)$.

\item Moreover, for $k<0$, we can actually take $\epsilon$ and $\delta$ to be small enough and extend this coupling for all $t\ge0$ such that $d(X_{t},Y_{t})\ge\rho_{0}$.

\item In the case $k\ge0$, we can find a small $\epsilon>0$ such that for any points $x_{0},y_{0}$ with $d(x_{0},y_{0})\le \epsilon$, there is a co-adapted coupling of Brownian motions $X_{t},Y_{t}$ with $X_{0}=x_{0}$ and $Y_{0}=y_{0}$ such that $d(X_{t},Y_{t})=\rho(t)$ for all $t\in[0,T)$.
\end{enumerate}

\end{corollary}

Essentially, one has to follow the same argument as in the proof of the Theorem, the only difference being that we need to replace $d(x,y)$, $U$, $V$ and the existence of the map $O_{U,V}$ satisfying \eqref{e7:15} and \eqref{e7:16} with $\rho_{t}$, $U_{t}$, $V_{t}$ and one of a map $O_{t}$ such that
\[
\begin{split}
&\langle U_{t}e_{i},\dot{\gamma}_{\pi U_{t},\pi V_{t}}  \rangle_{\pi U_{t}}-\langle V_{t}O_{t}e_{i},\dot{\gamma}_{\pi U_{t},\pi V_{t}}  \rangle_{\pi V_{t}}=0  \\
&\sum_{i=1}^{N}\mathcal{I}(J_{i},J_{i})=-\rho(t).
\end{split}
\]

We would like to point out that this in agreement with our results obtained in \cite{PP} the case of Euclidean spaces and spheres where we actually get a complete characterization of all coupling for which the distance function is deterministic.

\section{Refinements and Comments}\label{s:10}

The proof of Theorem~\ref{t:7} spreads on several pages, and some comments on it are in order. The first observation is that the conditions imposed are essential for the construction. For example the positivity of the injectivity radius is needed for the local construction.  The Ricci curvature bounded from below insures the non-explosion of the Brownian motion on one hand, and on the other hand it is important in the estimate of the exit times employed in the proof of the global existence of the coupling and also for the estimates involving the index form from \eqref{et:1}.

That the sectional curvature is bounded from above does not seem to be optimal even though it is an important piece in the proof of the existence of the coupling via the index form comparison on $M$ with the index form of a sphere.   Geometrically, we certainly need to make sure that the Brownian motions we try to couple do not get trapped in regions of extremely high sectional curvature  where the Brownian motions tend to get close to one another.    It seems though that the optimal condition would be that the injectivity radius of the manifold is positive.   However this certainly requires a different argument from the one provided here.

Another aspect is that the global existence of the choice of the map $O_{U,V}$ is tied to the existence of a smooth choice of an orthonormal frame on $M$.   On an arbitrary Riemannian manifold this can be done only locally and this is why we had to go one more step,  from the local existence of the coupling to its global existence.   There are though a few cases when the existence can be proved globally, one of which is the case of surfaces.   In this case, for any two points $x,y$ not at each other cut-locus, there is a single perpendicular direction to the geodesic joining $x$ and $y$.  Using this we can show that there is a global choice of $O_{U,V}$ as long as $\pi U,\pi V$ are not at each other cut-locus.

Another case in which we can construct a global version of $O_{U,V}$ is the one in which $M$ is parallelizable, namely the tangent bundle is trivializable, or otherwise put, there exist vector fields $X_{1},X_{2},\dots X_{d}$ which are independent at each point.  This amounts to the existence of a global section of the orthonormal frame bundle.   It is for instance the case of $S^{3}$ and $S^{7}$ and also of any Lie group with the left or right invariant metric.

The couplings we constructed in Theorem~\ref{t:7} are defined for all times $t\ge0$, and the conditions in \eqref{e7:0} were necessary in the proof.   There is however a case when the injectivity and upper bound on the sectional condition can be dispensed of if one only needs the coupling to be defined up to the first exit time of the coupling from a relatively compact set.   For completeness, we record the result here and use it in the next section.  The proof is the same as the one given above adjusted with a stopping time.

\begin{theorem}\label{t:8}  Let $M$ be a complete $d$-dimensional Riemannian manifold and $D\subset M$ a relatively compact open set of $M$ with a smooth boundary.   Then, there exists $\epsilon>0$ such that for any $x,y\in D$ with $d(x,y)<\epsilon$, there exist a shy coupling of two Brownian motions on $M$ starting at $x$ and $y$, defined up to the first exit time of either of the processes from $D$.

If in addition $Ric\ge 0$, there also exists a fixed-distance coupling Brownian motions on $M$ starting at $x$ and $y$, defined up to the first exit time of either of the processes from $D$.
\end{theorem}

The suggestion given by Kendall in \cite[Section 4]{Kendall} for the construction of the shy coupling is to use a form of perverse coupling (in the terminology of \cite{Kendall}).  However, this is not sufficient to get the fixed distance coupling.  Particularly this is very clearly illustrated in the case of surfaces. Indeed,  since the dimension is $2$,  we have just one dimension left in the orthogonal to the geodesic joining $X_{t}$ and $Y_{t}$ and then there are essentially only two choices of an orthogonal map from $T_{x}$ to $T_{y}$ (for $x,y$ not at each other cut-locus) which preserves the geodesic direction.  One choice is the one in which in the perpendicular direction to the geodesic, the particles move in the same direction which gives the mirror coupling or in the opposite directions which gives the perverse couplings.  None of these give the fixed distance coupling.

Another point is that one can get a shy coupling using stochastic flows.   In short, the idea is to impose conditions such that the flow stays a Brownian motion and this can be done if the direction in the Cameron-Martin space satisfies a certain ode.   If the initial value of this direction is non-zero everywhere then we obtain a weak form of shy coupling.  See for details \cite{Elton,MR2541273}.

Though we have dealt with a coupling of two Brownian motions, we can actually construct a family of Brownian motions indexed by some set.  For instance, given $x,y$ to points in $M$, the construction in \cite[Theorem 10.37]{Stroock2}, gives a family of Brownian motions $X_{t}^{s}$ for $s$ running in $[0,d(x,y)]$ such that $\frac{d}{ds}X^{s}_{t}\le e^{-kt/2}d(x,y)$.

What we can do is the following.  Take $\epsilon>0$ small enough and then we can construct a family $X_{t}^{x}$ and $Y_{t}^{y}$ where $x,y\in M$ with distance $d(x,y)<\epsilon$ such that at least for small time $t\in[0,\delta]$ we get that $d(X_{t}^{x},Y_{t}^{y})=e^{-kt/2}d(x,y)$.   In the case $k\ge0$ we obtain in fact that the coupling is defined for all $t\ge0$.   The whole idea is that in our local construction of Theorem~\ref{t:100}, the choice of the orthogonal matrix boils down to choosing the angle $\alpha$ for the matrix $B_{x,y}$ in \eqref{e7b:10}.

\section{Applications}\label{s:11}

\subsection{The Brownian Lion and the Man}

We started this paper with the Lion and the Man and we close it with a simple interpretation of the results in this language.  Assume we have a Riemannian manifold $M$ satisfying the conditions of Theorem~\ref{t:7}.   Then, given a Brownian Lion running on $M$, Theorem \ref{t:7} assures that there is a strategy for the Brownian Man which keeps him at a safe positive distance from the Lion for all times.

In addition, if the Ricci is non-negative, then the Brownian Man can choose a strategy which keeps him at fixed distance from the Brownian Lion.  This must be particularly frustrating for the Lion especially if they start relatively close to each other.

Theorem \ref{t:7} also shows that if the Ricci curvature is bounded below by a positive constant, then given a Brownian Man, the Brownian Lion has a strategy which will bring him arbitrarily close to its meal.

\subsection{Lower Bounds on Ricci Curvature} As we pointed out in the introduction,  \cite[Corollary 1.4]{ReSt} shows that one can characterize the condition $Ric \ge k$ in terms of couplings.   We now have an optimal version of it which is formally put here.

\begin{corollary}  Assume $M$ is a complete Riemannian manifold.    Then the following two statements are equivalent.
\begin{enumerate}
\item $Ric_{x} \ge k$ for all $x\in M$.
\item For any point $z\in M$, there exist $r_{z},\delta_{z}>0$ such that for any $x,y\in B(z,r_{z})$ we can find a Markovian coupling of Brownian motions $X_{t},Y_{t}$ starting at $x,y$ with the property that
\[
d(X_{t},Y_{t})=e^{-kt/2}d(x,y) \text{ for } 0\le t\le \delta_{z}\wedge \zeta_{z}
\]
where $\zeta_{z}$ is the first time either $X_{t}$ or $Y_{t}$ exit the ball $B(z,r_{z})$.
\end{enumerate}
\end{corollary}

As a clarification, $X_{t},Y_{t}$ need to be defined up to the exit time from the ball $B(z,r_{z})$ or up to $\delta_{z}$, whichever comes up first.

\begin{proof}  The implication $1)\Longrightarrow 2)$ follows from Theorem~\ref{t:8}.  For the reverse implication we follow the same lines as in \cite{ReSt}, particularly the implication (x)$\Longrightarrow$(i) and we will sketch only the main differences.

Instead of considering the heat kernel of the Laplacian on the manifold we consider the heat kernel $p_{t}(x,y)$  of half the Laplacian on $B(z,r_{z})$ with the Dirichlet boundary conditions and its corresponding action $(p_{t}f)(x)=\int_{B(z,r_{z})}p_{t}(x,y)f(y)dy$.   Using this we can prove that condition 2) implies for any points $x,y\in B(z,r_{z})$ and any compactly supported function $f$ on $B(z,r_{z})$,
\[
p_{t}f(x)-p_{t}f(y)=\E[f(X_{t\wedge \zeta_{z}})-f(Y_{t \wedge \zeta_{z}})]\le |\nabla f|_{B(z,r_{z})}d(x,y)\E[e^{-k(t\wedge \zeta_{z})/2}]
\]
from which one immediately gets by letting $y$ approach $x$ that
\[
|\nabla p_{t}f(x)|\le |\nabla f|_{B(z,r_{z})}\E[e^{-k(t\wedge \zeta_{z})/2}].
\]
Now, with very little changes in the argument of the implication (v)$\Longrightarrow$(i) from \cite{ReSt}, if $Ric_{z}(v,v)<k$ at some point $z$ for some $v$ we arrive at the following conclusion
\[
k\E\left[1-\frac{t\wedge \zeta_{z}}{t}\right]\ge \epsilon +o(1)
\]
for some $\epsilon >0$.   This certainly leads to a  contradiction as we let $t\to 0$.    \qedhere
\end{proof}

\section*{Acknowledgements} We want to thank Wilfrid Kendall and Krzysztof Burdzy for several interesting discussions on the existence of fixed-distance coupling on the sphere which took place in the summer of 2009 while the first author visited University of Warwick.  This motivated us to undertake, extend and complete this program on manifolds.

We also want to thank Rob Neel for pointing to us that we do not have to extend the coupling at the cut-locus and that it suffices to let the Brownian motions run independently near the cut-locus.  Also thanks are in place to Elton P. Hsu for a discussion around Markovian couplings and Marc Arnaudon for several comments and references.

We would also like to thank the reviewer of this paper for the careful reading of the manuscript and for the suggestions which lead to an improvement of the present version.

\providecommand{\bysame}{\leavevmode\hbox to3em{\hrulefill}\thinspace}
\providecommand{\MR}{\relax\ifhmode\unskip\space\fi MR }
\providecommand{\MRhref}[2]{%
  \href{http://www.ams.org/mathscinet-getitem?mr=#1}{#2}
}
\providecommand{\href}[2]{#2}


\begin{thebibliography}{10}

\bibitem{MR2790368}
Marc Arnaudon, Kol{\'e}h{\`e}~Abdoulaye Coulibaly, and Anton Thalmaier,
  \emph{Horizontal diffusion in {$C^1$} path space}, S\'eminaire de
  {P}robabilit\'es {XLIII}, Lecture Notes in Math., vol. 2006, Springer,
  Berlin, 2011, pp.~73--94.

\bibitem{MR2215664}
Marc Arnaudon, Anton Thalmaier, and Feng-Yu Wang, \emph{Harnack inequality and
  heat kernel estimates on manifolds with curvature unbounded below}, Bull.
  Sci. Math. \textbf{130} (2006), no.~3, 223--233.

\bibitem{Burdzy-Benjamini}
Itai Benjamini, Krzysztof Burdzy, and Zhen-Qing Chen, \emph{Shy couplings},
  Probab. Theory Related Fields \textbf{137} (2007), no.~3-4, 345--377.

\bibitem{MR2502429}
Anca-Iuliana Bonciocat and Karl-Theodor Sturm, \emph{Mass transportation and
  rough curvature bounds for discrete spaces}, J. Funct. Anal. \textbf{256}
  (2009), no.~9, 2944--2966. \MR{2502429 (2010i:53066)}

\bibitem{Burdzy-Kendal}
Maury Bramson, Krzysztof Burdzy, and Wilfrid Kendall, \emph{Shy couplings,
  {$\rm CAT(0)$} spaces, and the {L}ion and {M}an}, Ann. Probab. \textbf{41}
  (2013), no.~2, 744--784. \MR{3077525}

\bibitem{DoC}
Manfredo Perdigao~Do Carmo, \emph{Riemannian geometry}, Birkh\"auser, Boston,
  1992.

\bibitem{Cheeger}
Jeff Cheeger and David~G. Ebin, \emph{Comparison theorems in {R}iemannian
  geometry}, North-Holland Publishing Co., Amsterdam, 1975, North-Holland
  Mathematical Library, Vol. 9.

\bibitem{MR3021521}
Hee~Je Cho and Seong-Hun Paeng, \emph{Ollivier's {R}icci curvature and the
  coloring of graphs}, European J. Combin. \textbf{34} (2013), no.~5, 916--922.

\bibitem{CranstonJFA}
Michael Cranston, \emph{Gradient estimates on manifolds using coupling}, J.
  Funct. Anal. \textbf{99} (1991), no.~1, 110--124.

\bibitem{Elworthy}
K.~D. Elworthy, \emph{Stochastic differential equations on manifolds}, London
  Mathematical Society Lecture Note Series, vol.~70, Cambridge University
  Press, Cambridge, 1982.

\bibitem{Elworthy2}
\bysame, \emph{Stochastic differential equations on manifolds}, Probability
  towards 2000 ({N}ew {Y}ork, 1995), Lecture Notes in Statist., vol. 128,
  Springer, New York, 1998, pp.~165--178.

\bibitem{MR2989449}
Matthias Erbar and Jan Maas, \emph{Ricci curvature of finite {M}arkov chains
  via convexity of the entropy}, Arch. Ration. Mech. Anal. \textbf{206} (2012),
  no.~3, 997--1038. \MR{2989449}

\bibitem{Elton}
Elton~P. Hsu, \emph{Stochastic analysis on manifolds}, Graduate Studies in
  Mathematics, vol.~38, American Mathematical Society, Providence, RI, 2002.

\bibitem{MR2541273}
Elton~P. Hsu and Cheng Ouyang, \emph{Quasi-invariance of the {W}iener measure
  on the path space over a complete {R}iemannian manifold}, J. Funct. Anal.
  \textbf{257} (2009), no.~5, 1379--1395. \MR{2541273 (2010h:58054)}

\bibitem{HsuSturm}
Elton~P. Hsu and Karl-Theodor Sturm, \emph{Maximal coupling of {E}uclidean
  {B}rownian motions}, Commun. Math. Stat. \textbf{1} (2013), no.~1, 93--104.

\bibitem{Kendall5}
Wilfrid~S. Kendall, \emph{Nonnegative {R}icci curvature and the {B}rownian
  coupling property}, Stochastics \textbf{19} (1986), no.~1-2, 111--129.

\bibitem{Kendall}
\bysame, \emph{Brownian couplings, convexity, and shy-ness}, Electron. Commun.
  Probab. \textbf{14} (2009), 66--80.

\bibitem{Kuw2}
Kazumasa Kuwada, \emph{On uniqueness of maximal coupling for diffusion
  processes with a reflection}, J. Theoret. Probab. \textbf{20} (2007), no.~4,
  935--957.

\bibitem{Kuw1}
\bysame, \emph{Characterization of maximal {M}arkovian couplings for diffusion
  processes}, Electron. J. Probab. \textbf{14} (2009), no. 25, 633--662.

\bibitem{MR2872958}
Yong Lin, Linyuan Lu, and Shing-Tung Yau, \emph{Ricci curvature of graphs},
  Tohoku Math. J. (2) \textbf{63} (2011), no.~4, 605--627.

\bibitem{Lin-Rog}
Torgny Lindvall and L.~C.~G. Rogers, \emph{Coupling of multidimensional
  diffusions by reflection}, Ann. Probab. \textbf{14} (1986), no.~3, 860--872.
  \MR{841588 (88b:60179)}

\bibitem{Littlewood}
John~E. Littlewood, \emph{Littlewood's miscellany}, Cambridge University Press,
  Cambridge, 1986, Edited and with a foreword by B{\'e}la Bollob{\'a}s.

\bibitem{MR2480619}
John Lott and C{\'e}dric Villani, \emph{Ricci curvature for metric-measure
  spaces via optimal transport}, Ann. of Math. (2) \textbf{169} (2009), no.~3,
  903--991.

\bibitem{MR2484937}
Yann Ollivier, \emph{Ricci curvature of {M}arkov chains on metric spaces}, J.
  Funct. Anal. \textbf{256} (2009), no.~3, 810--864.

\bibitem{PP}
Mihai~N. Pascu and I.~Popescu, \emph{Couplings of brownian motions of
  deterministic distance in the euclidean space and on the sphere}, preprint
  (2015).

\bibitem{Spivak}
Michael Spivak, \emph{A comprehensive introduction to differential geometry.
  {V}ol. {IV}}, second ed., Publish or Perish, Inc., Wilmington, Del., 1979.
  \MR{532833 (82g:53003d)}

\bibitem{Stroock2}
Daniel~W. Stroock, \emph{An introduction to the analysis of paths on a
  {R}iemannian manifold}, Mathematical Surveys and Monographs, vol.~74,
  American Mathematical Society, Providence, RI, 2000.

\bibitem{MR2237206}
Karl-Theodor Sturm, \emph{On the geometry of metric measure spaces. {I}}, Acta
  Math. \textbf{196} (2006), no.~1, 65--131.

\bibitem{MR2237207}
\bysame, \emph{On the geometry of metric measure spaces. {II}}, Acta Math.
  \textbf{196} (2006), no.~1, 133--177.

\bibitem{ReSt}
Max-K. von Renesse and Karl-Theodor Sturm, \emph{Transport inequalities,
  gradient estimates, entropy, and {R}icci curvature}, Comm. Pure Appl. Math.
  \textbf{58} (2005), no.~7, 923--940.

\end{thebibliography}
\end{document}